\documentclass[12pt]{amsart}

\usepackage{psfrag}

\numberwithin{figure}{section}

\newtheorem{theorem}{Theorem}[section]
\newtheorem{lemma}[theorem]{Lemma}
\newtheorem{proposition}[theorem]{Proposition}
\newtheorem{corollary}[theorem]{Corollary}

\theoremstyle{definition}
\newtheorem{definition}[theorem]{Definition}
\newtheorem{remark}[theorem]{Remark}
\newtheorem{example}[theorem]{Example}

\numberwithin{equation}{section}
\newcommand{\beq}{\begin{equation}}

\newcommand{\eq}[2]{\begin{equation}\label{#1} #2\end{equation}}
\newcommand{\eqn}[1]{\begin{equation*}#1 \end{equation*}}

\newcommand{\eqs}[2]{\begin{equation}\label{#1}\begin{split} #2\end{split}\end{equation}}

\newcommand{\eqns}[1]{\begin{equation*}\begin{split} #1\end{split}\end{equation*}}

\newcommand{\beqa}{\begin{eqnarray}}
\newcommand{\eeqa}{\end{eqnarray}}

\newcommand{\beaa}{\begin{eqnarray*}}
\newcommand{\ben}{\begin{eqnarray*}}
\newcommand{\eaa}{\end{eqnarray*}}
\newcommand{\een}{\end{eqnarray*}}

\newcommand{\R}{\hbox{\rm I\kern-.13em  R}}
\newcommand{\Rn}{\R^n}
\newcommand{\dist}{{\rm dist}\,}
\newcommand{\xs}[1]{x_{#1}}
\newcommand{\ys}[1]{y_{#1}}

\newcommand{\iq}[1]{\quad\hbox{#1}\quad}
\renewcommand{\i}[1]{\text{ \ #1\ }}
\newcommand{\set}[1]{\{#1\}}
\newcommand{\Lp}{\left(}
\newcommand{\Rp}{\right)}

\newcommand{\ip}[1]{\left\langle #1\right\rangle}

\newcommand{\sign}{{\rm sign}}
\newcommand{\ra}{\rightarrow}

\newcommand{\reff}[1]{(\ref{#1})}

\newcommand{\no}[1]{\left\|#1\right\|}
\newcommand{\dt}{\frac {d\ }{dt}}

\newcommand{\nop}[1]{\|#1\|_p}
\newcommand{\da}{\downarrow}
\newcommand{\noe}[1]{|#1|}

\newcommand{\z}{l}
\newcommand{\ls}[1]{\z_{#1}}
\newcommand{\bls}[1]{\bl_{#1}}
\newcommand{\bl}{\bar \z}
\newcommand{\lst}{\z_*}
\newcommand{\blst}{\bl_*}
\newcommand{\be}{\bar e}
\newcommand{\by}{\bar y}

\newcommand{\li}{w}

\newcommand{\proofof}[1]{\noindent {\it Proof of #1}}

\title{The Problem of Two Sticks}
\author{Luis A. Caffarelli\\ Michael G. Crandall }
\address{Department of Mathematics, University of Texas at Austin,  Austin, TX, 78712}
\email{caffarel@math.utexas.edu}

\address{Department of Mathematics, University of California,  Santa Barbara,
Santa Barbara, CA 93106}
\email{crandall@math.ucsb.edu}
 
\makeindex

\begin{document}

\thispagestyle{empty} 
\centerline{\scshape Cover Page}\vskip.2in
\large
\noindent Title:  The Problem of Two Sticks  \footnote{MSC 2000 Subject Classification, Primary 46B20, 52A21.} \newline

\noindent Authors:\newline
\indent Luis A. Caffarelli\\
 \indent Department of Mathematics, University of Texas at Austin,\\
\indent    Austin, TX, 78712\\  \indent caffarel@math.utexas.edu \\
\indent 512-471-3160
\vskip.1in

 Michael G. Crandall

 Department of Mathematics, University of California,\\ \indent  Santa Barbara,
Santa Barbara, CA 93106 \\
\indent   crandall@math.ucsb.edu\\
\indent 805-964-3256

\newpage
\setcounter{page}{1}

\normalsize
\begin{abstract} Let $\z=[\ls0,\ls1]$ be the directed line segment from $\ls0\in\Rn$ to $\ls1\in\Rn.$   Suppose $\bl=[\bls0,\bls1]$ is a second segment of equal length such that $\z, \bl$ satisfy the ``two sticks condition": $\no{\ls1-\bls0}\ge \no{\ls1-\ls 0}, \no{\bls1-\ls0}\ge \no{\bls1-\bls 0}.$ Here $\no{\cdot}$ is a norm on $\Rn.$  We explore the manner in which $\ls1-\bls1$ is then constrained when assumptions are made about ``intermediate points" $\lst\in \z, \blst \in \bl.$  Roughly speaking, our most subtle result constructs parallel planes separated by a distance comparable to $\no{\lst-\blst}$ such that $\ls 1-\bls1$ must lie between these planes,   provided that $\no{\cdot}$ is ``geometrically convex" and ``balanced", as defined herein.  The standard $p$-norms are shown to be geometrically convex and balanced. Other results estimate $\no{\ls1-\bls1}$ in a Lipschitz or H\"older manner by $\no{\lst-\blst}$.  All these results have implications in the theory of eikonal equations, from  which this ``problem of two sticks" arose.
\end{abstract}
\maketitle

\centerline{{\scshape Introduction}}\vskip.1in
\setcounter{section}{1}
The origin of the ``problem of two sticks," which we are about to describe, lies in the theory of eikonal equations.   Roughly speaking, the results of  Caffarelli and Crandall \cite{cc} rely on knowledge of how the endpoints  of ``rays" of the distance function to some set, as measured in a norm $\no{\cdot},$ that emanate from points in the set and  pass through a common tiny ball in the interior of the region of differentiability of the distance function are constrained.  We provide a variety of results that speak to this issue. In particular, the crown jewel of our results, Corollary \ref{cstrip} below,  implies that the endpoints must lie between parallel planes which are separated by a distance comparable to the radius of the ball. 

 The ingredients of the problem of two sticks are a norm $\no{\cdot}$ on $\Rn$ and 
 two ``sticks"
\eqn
{
\z=[\ls 0, \ls 1], \ \bl=[\bls 0,\bls1], 
}
where $[\ls 0, \ls 1]$ denotes the directed line segment from $\ls0$ to $\ls1\in\Rn.$ Sometimes we regard $\z$ as a set, as when we write $x\in \z,$ or $x\in [\ls 0, \ls 1],$ but $[x,y]$ has an ``initial"  point $x$ and a ``terminal" point $y.$ We assume throughout this paper  that the sticks satisfy the ``two sticks condition"
\eq{ts}
{
\no{\ls 1-\bls 0}\ge \no{\ls1-\ls0}\iq{and}\no{\bls 1-\ls 0}\ge \no{\bls1-\bls0}.
}

To emphasize our remarks about the ordering of the endpoints of the sticks, observe that if $\ls0=\bls0,$ then \reff{ts} is satisfied for any choice whatsoever of $\ls 1, \bls 1.$
  However, if $\ls1=\bls1=0,$ then \reff{ts} is satisfied iff $\no{\ls0}=\no{\bls{0}}.$  In particular, in general, interchanging the initial and terminal points of sticks $l, \bl$ which satisfy \reff{ts} can lead to sticks which do not satisfy \reff{ts}.  For further remarks about the nature of the two sticks condition, see Section \ref{tsdf}, where we explain its relationship to nearest point mappings and distance functions.
  
Usually we will assume the sticks are of equal length $L:$
\eq{sl}
{
\no{\ls 1-\ls 0}=\no{\bls1-\bls0}=L.
}

 Assume that a point of $l$ is ``close" (to be quantified) to a point of $\bl,$ each point being somewhere away from the endpoints of the stick in which it lies.  The two sticks problem is then to obtain information about $\ls 1-\bls 1.$ In what manner is it constrained?

  For example, suppose that \reff{ts} holds, \reff{sl} holds with $L=1$ (a normalization),  the sticks intersect  at a  point $z\in l\cap\bl,$ and $z$ is not an endpoint of either stick.   Then one has 
\eqs{coin}
{
1\le\no{\ls1-\bls0}\le\no{\ls1-z}+\no{z-\bls0},
\ 1\le\no{\bls1-\ls0}\le\no{\bls1-z}+\no{z-\ls0},
}
which, when added, give
\eqns
{
2\le \no{\ls1-\bls0}+\no{\bls1-\ls0}\le \no{\ls1-z}&+\no{z-\ls0}+\no{\bls1-z}+\no{z-\bls0}
\\&
=\no{\ls 1-\ls 0}+\no{\bls 1-\bls 0}=2.
}
It follows that each inequality in \reff{coin} must be an equality.  If the norm is strictly convex (see Section \ref{mgn}), this entails the existence of positive constants $\alpha, \beta$ such that 
\eqns
{
\ls1-z=\alpha(z-\bls0),\ \bls1-z=\beta(z-\ls0).
}

To continue, since $z$ is an intermediate point of both sticks, each of which has length 1, the above implies  
$
\ls1-\ls0=\bls1-\bls0. 
$
However,  a moment's thought reveals that if the directions of unit length sticks are the same and they have a common intermediate point, they cannot satisfy the two sticks condition without being identical, that is $\ls1=\bls1,$ $\ls0=\bls0.$ We go a bit further with this, now allowing, for example, $z=\ls 1.$ A picture quickly reveals that the two sticks condition then fails unless $\ls 1=\bls 1.$ We can no longer assert that $\ls 0=\bls0,$ but surely $\ls 1=\bls 1$ still holds.  Similarly, if $z=\ls 0,$ then $\ls 0=\bls 0,$ but we can no longer assert that $\ls 1=\bls 1.$ In all, $\ls 1=\bls 1$ holds if the sticks have a common point, so long as that common point is not $\ls 0=\bls 0.$  It follows that given a collection $\set{l^i, i\in {\mathcal I}}$ of sticks of unit length, indexed here by some index set ${\mathcal I},$  which pairwise satisfy the two sticks condition, then 
the mapping from the set of all intermediate points from the family to terminal points of sticks in which they lie is well defined.   It is properties of this mapping which are called on in \cite{cc}. 

Using the simple result already noted,  straightforward compactness arguments show that if the norm is strictly convex,  $0<\varepsilon,$ and $0<t\le 1,$ then there is a $\delta_0=\delta_0(\varepsilon,t)>0$ such that 
\eqs{close}
{
\lst\in \z, \blst\in\bl,\ \ t\le\no{\lst-\ls0}, \no{\blst-\bls0},
}
and
\eq{cdelta}
{
\no{\lst-\blst}\le \delta_0
}
imply $\no{\ls1-\bls1}\le \varepsilon.$
That is, the mapping referred to in the preceding paragraph is continuous.

These remarks are not strong enough for our intended applications to eikonal equations, owing to the general behavior of $\delta_0$ as a function of $\varepsilon.$  Thus 
we prove a hierarchy of variants  under additional conditions. Indeed, in the case of the Euclidean norm on $\Rn,$ when the two sticks and equal length conditions are satisfied as well as \reff{close},  the mapping associated with \reff{close} here,  that is $\lst\mapsto \ls1, \blst\mapsto \bls1,$  is Lipschitz continuous; in fact, Corollary \ref{clip} below implies that then 
\eqn
{
\no{\ls1-\bls1}\le \frac {2}t\no{\lst-\blst}. 
}
This Lipschitz continuity also holds for norms  which are ``2 uniformly smooth and 2 uniformly convex" (see Section \ref{pucqus} for the definition). This is a special case of the main result of Section \ref{pucqus}, which states that if the norm is $p$-uniformly convex and $q$-uniformly smooth, then the mapping is H\"older continuous with exponent $q/p.$ 

The results of Section \ref{pucqus} apply to the $p$-norms on $\Rn,$ that is $\no{\cdot}=\nop{\cdot},$ where
\eq{lp}
{
\nop{x}:=\Lp\sum_{i=1}^n|\xs i|^p\Rp^{1/p},
}
in the range $1<p<\infty.$ Indeed, $\nop{\cdot}$ is 2-uniformly smooth and $p$-uniformly convex for $2\le p<\infty, $ and it is 2-uniformly convex and $p$-uniformly smooth for $1<p\le 2.$  In Section \ref{cexlp} we  provide examples to show that the H\"older continuity established for the $\nop{\cdot}$ cases is asymptotically an optimal modulus of continuity, up to constants.

However,  the H\"older continuity obtained in the $\nop{\cdot}$ case is not always sufficient for the purposes of \cite{cc}, even if the modulus is optimal. This deficiency led us to the notion of  norms which are ``geometrically convex," as introduced in Section \ref{spc}. For geometrically convex norms, which are also ``balanced", it is shown in Section \ref{tspc} that, roughly speaking, if $\delta_0$ is sufficiently small, then  \reff{close}, \reff{cdelta} imply that 
$\ls1-\bls1$ is confined between two parallel planes which are separated by a distance which is an estimable multiple of $\delta_0$. This is, of course, not a ``modulus of continuity" result; it is more subtle.  It is another task to verify that the $p$-norms are geometrically convex and balanced, and this we do in Section \ref{vgc}.  

We begin with the Euclidean case, after some remarks about the two sticks problem and distance functions. In this regard, it is clear that the problem of two sticks is  related to properties of nearest point mappings onto convex sets, and we  recognized that the results of Section \ref{pucqus} were likely to hold via papers concerning this issue.  These include, for example, 
B. Bj\"ornestal \cite{bjorn}, Y. Alber \cite{alber} and C. Li, X. Wang and W. Yang \cite{lwy}.  However, our Section \ref{pucqus} is short and  self-contained; correspondingly, our constants are not sharp.  

In contrast, the results and notions of Sections \ref{spc},  \ref{sc}, \ref{tspc} and \ref{vgc} are not suggested by other literature of which we are aware.  

As this entire paper could be made essentially self-contained, we have done so. Thus in the first part of Section \ref{mgn} we have presented some well-known elementary material with perhaps a different spirit than is usual; in particular, we do not use dual spaces or dual norms explicitly anywhere in this work.  

\tableofcontents
\section{\label{tsdf}Two Sticks and the Distance Function}

Suppose that $C\subset\Rn$ and $\ls 1, \bls 1\in\Rn.$ Suppose that $\ls 0, \bls 0\in C$ and 
\eq{nearest}
{
\no{\ls 1-\ls 0}\le \no{\ls 1-x}, \no{\bls 1-\bls 0}\le \no{\bls 1-x}\iq{for}x\in C.
}
Then $\ls 0$ is a point of $C$ which is as close to $\ls 1$ as any other point of $C,$ etc. Choosing $x=\bls0$ in the first inequality of \reff{nearest} and $x=\ls0$ in the second, we see that $l=[\ls0,\ls 1],$ $\bl=[\bls0,\bls 1]$  satisfy the two sticks condition. Conversely, if $l, \bl$ satisfy the two sticks condition and $C=\set{\ls0,\bls0},$ we have \reff{nearest}.  Moreover, we have, in both cases,
\eq{distf}
{
\no{\ls1 -\ls 0}=\dist(\ls1,C), \no{\bls1-\bls0}=\dist(\bls1,C),
}
where $\dist(x,C)$ is the distance, as measured by $\no{\cdot},$ from $x$ to $C.$  If we add the equal length condition, we are assuming these distances are equal.  Thus the study of the two sticks problem is a kind of atomization of the study of ``rays" of  distance functions, wherein lies its connection to Hamilon-Jacobi equations. 

Continuing in this line,  the notation
 \eq{lt}
{
\ls t:=(1-t)\ls0+t\ls1,\ \bls t:=(1-t)\bls0+t\bls1,
}
is used in the next remarks.  Note that we use $A:=B$ to indicate that $A$ is defined to be $B.$ 

 First, if $l, \bl$ satisfy the two sticks condition, then so do $[\ls 0, \ls t], [\bls 0, \bls 1]$ for $0\le t\le 1.$ To see this, merely note that 
\eqn
{
\no{\ls 1-\ls 0}=\no{\ls 1-\ls t}+\no{\ls t-\ls 0}\le\no{\ls1-\bls0} 
}
implies
\eq{ltt}
{
\no{\ls t-\ls 0}\le\no{\ls1-\bls0}-\no{\ls 1-\ls t}\le \no{\ls t-\bls 0}.
}
Iterating this remark, if $0\le s, t\le 1,$ then $[\ls 0, \ls t], [\bls 0, \bls s]$ also satisfy the two sticks condition.

When we add an equal length condition, say
\eqn
{
\no{\ls 1-\ls 0}=\no{\bls 1-\bls 0}=L,
}
there is an additional symmetry.  Observe then that 
\eqns
{
&\no{\ls 1-\ls 0}\le \no{\ls 1-\bls 0}\implies \no{\bls 1-\bls 0}\le \no{\ls 1-\bls 0},
\\
&\no{\bls 1-\bls 0}\le \no{\bls 1-\ls 0}\implies \no{\ls 1-\ls 0}\le \no{\bls 1-\ls 0};
}
that is the sticks $[\ls 1, \ls 0], [\bls 1, \bls 0]$ obtained by switching initial and terminal points also satisfy the two sticks condition. It follows from this  that the equal length condition and the previous discussion guarantee that each line below implies the next when $0\le t,s\le1:$
\eqs{flip}
{
{\rm (i)\ \ } &[\ls 0, \ls 1], [\bls 0, \bls 1] \iq{satisfy the two sticks and equal length conditions.}
\\
{\rm (ii)\ } &
[\ls 1, \ls 0], [\bls 1, \bls 0] \iq{satisfy the two sticks and equal length conditions.}
\\
{\rm (iii)\ } &
[\ls 1, \ls t], [\bls 1, \bls t] \iq{satisfy the two sticks and equal length conditions.}
\\
{\rm (iv)\ } &
[\ls t, \ls 1], [\bls t, \bls 1] \iq{satisfy the two sticks and equal length conditions.}
\\
{\rm (v)\ } &
[\ls t, \ls s], [\bls t, \bls s] \iq{satisfy the two sticks and equal length conditions. }
}
Indeed, note that if $0\le t\le s\le 1,$ and
$
l=[\ls0, \ls 1], \tilde l=[\ls1,\ls 0], \hat l=[\ls t,\ls 1] 
$, then
$ {\tilde l}_{1-t}= \ls{t}, \ {\hat l}_ {(s-t)/(1-t)}= \ls s. 
$

In particular, we note for later use that, via \reff{flip} (v),
\eq{flipa}
{
\no{\bls s-\ls t}\ge \no{\bls s-\bls t}, \ \no{\ls s-\bls t}\ge \no{\ls s -\ls t}\iq{for} 0\le t,s\le 1.
}

\section{Cases with Lipschitz or H\"older Continuity}

In this section we first treat the  Euclidean case.  Then we turn to the ``$p$-uniformly convex, $q$-uniformly smooth" case.  The Euclidean (or, more generally, Hilbert) case is also an example in which $p=q=2$. However, as is usual, it is clean and elegant in comparison to its generalization, and deserves to be singled out. 
\subsection{Two sticks in the Euclidean case}\label{tseuc}

We will denote the Euclidean norm by $\noe{\cdot};$ 
\eq{eucn}
{
\noe x:=\sqrt{\ip{x,x}},
}
where 
\eq{dip} 
{
\ip{x,y}:=\sum_{j=1}^n\xs j\ys j
}
is the Euclidean inner-product.

We begin with estimates valid for sticks $\z, \bl$ which satisfy the two sticks condition \reff{ts}, but which do not necessarily have the same length. The notation \reff{lt} is employed. The next result is well known. 

\begin{proposition}\label{mono} Let $\z, \bl$ satisfy the two sticks condition \reff{ts}. Then 
\eq{emono}
{
\ip{\ls1-\bls1,\ls0-\bls0}\ge 0.
}
In consequence, for $0\let\le1,$
\eq{lipa}
{
(1-t)^2|\ls0-\bls0|^2+ t^2|\ls1-\bls1|^2\le |\ls t-\bls t|^2.
}
\end{proposition} 
\begin{proof} The relation \reff{emono} follows from adding the extremes in the relations
\eqns
{
|\ls1-\ls0|^2+|\ls0-\bls0|^2+2\ip{\ls1-\ls0,\ls0-\bls0}=&|(\ls1-\ls0)+(\ls0-\bls0)|^2=|\ls1-\bls0|^2\ge |\ls1-\ls0|^2,\\
|\bls1-\bls0|^2+|\ls0-\bls0|^2+2\ip{\bls1-\bls0,\bls0-\ls0}=&|(\bls1-\bls0)+(\bls0-\ls0)|^2=|\bls1-\ls0|^2\ge |\bls1-\bls0|^2,
}
and simplifying the result. 

To verify \reff{lipa}, we use \reff{emono} to deduce 
\eqns
{
|\ls t-\bls t|^2&=|(1-t)(\ls 0-\bls 0)+t(\ls1-\bls 1)|^2\\
&=(1-t)^2|\ls 0-\bls 0|^2+2t(1-t)\ip{\ls0-\bls0,\ls1-\bls1}+t^2|\ls1-\bls 1|^2\\
&\ge (1-t)^2|\ls0-\bls0|^2+ t^2|\ls1-\bls1|^2,
}
which is \reff{lipa}. 
\end{proof}

\begin{remark} \label{lip} The relation \reff{lipa} shows that the terminal point $\ls1$ is a Lipschitz continuous function of the intermediate point $\ls t,$ $0<t\le 1,$ with Lipschitz constant $1/t,$  in any family of sticks which pairwise satisfy the two sticks condition.  Note again that if  $\ls0=\bls0$  then the two sticks condition is always satisfied, so  $0<t$ is necessary to have Lipschitz continuity.  
\end{remark}

If we add the equal length assumption \reff{sl}, the Lipschitz continuity may be extended to intermediate points $\ls s, \bls t,$ where $s\not=t.$

\begin{corollary}\label{clip} Let $l, \bl$ satisfy \reff{ts} and \reff{sl}. Then 
\eq{lipb}
{
|\ls1-\bls1|\le \frac 2t|l_s-\bls t|\iq{for} 0<t\le s\le 1.
}
\end{corollary}
\begin{proof}
Using \reff{lipa} 
\eq{ats}
{
|\ls1-\bls1|\le  \frac1t|\ls t-\bls t|.
}
First we assume that
\eq{ca}
{
|\ls t-\ls s |\le \frac 12 |\ls t-\bls t|.
}
Then, using \reff{ats}, 
\eq{enough}
{
|\ls s-\bls t |\ge |\bls t-\ls t|-|\ls t -\ls s|\ge  |\bls t-\ls t|-\frac 12 |\ls t-\bls t|=\frac 12|\bls t-\ls t|\ge \frac t 2|\ls 1-\bls 1|,
 }
 so \reff{lipb} holds.  If \reff{ca} does not hold, then we use \reff{flipa}, \reff{lipa}, to again conclude that 
\eq{cb}
{
\noe{\ls s-\bls t}\ge|\ls t-\ls s|\ge\frac 12 |\ls t-\bls t|\ge \frac t 2\noe{\ls 1-\bls 1}.
}
\end{proof}

 \begin{remark}\label{needel} If the equal length assumption is not satisfied, there is no Lipschitz estimate quite like \reff{lipb}.  To see this, let $n=1$ and take $\z= [0,1], \bl=[0,2], s=1, t=1/2.$ 
 \end{remark}
 
\subsection{\label{pucqus}Two Sticks in the $p$-Uniformly Convex, $q$-Uniformly Smooth Case}
 
In this section, $\no{\cdot}$ is a norm for which 
there are constants $0<A, B, q, p,$ with $1< q\le p,$ such that 
\eq{bsce}
{
 A\no{e-\be}^p\le  2-\no{e+\be} \iq{for} \no e=\no {\be}=1,}
that is, $\no{\cdot}$ is ``$p$-uniformly convex,"
and  
\eq{bsme}
{
\no{x+y}+\no{x-y}-2\no x\le \frac{B\ }{\no x^{q-1}}\no y^q\iq{for}x\not=0,
}
that is, $\no{\cdot}$ is ``$q$-uniformly smooth."  

\begin{remark}\label{ysm} Note that \reff{bsme} holds in general if it holds for $\no x=1.$ Moreover, if $\no x=1$ and \reff{bsme} holds for small $\no y,$ then it holds (with a different $B$) for all $y,$ as the left hand side is at most $2\no y$ and $q\ge 1.$ 
\end{remark}

\begin{remark} If $2\le p<\infty,$ the $p$-norm $\nop{\cdot}$ is $2$-uniformly smooth and $p$-uniformly convex, while if $1<p\le 2,$ it is 2-uniformly convex, and $p$-uniformly smooth.   The first assertion was proved by Clarkson \cite{jac} and the second by Hanner \cite{hanner}.   Regarding the more general and precise notions of ``modulus of convexity" and ``modulus of smoothness" and relations between them, see Lindenstrauss \cite{lin}. 
\end{remark}
\begin{proposition}\label{holderp}  Let \reff{bsce}, \reff{bsme} hold and $R>0.$ Then there is a constant $C=C(A,B,R,p,q)$ such that  if  $\z, \bl,$ satisfy the two sticks condition, have unit length,
and satisfy
\eq{param}
{
 \no{\ls1-\bls1}\le R,}
  then
  \eq{holder}
{
\no{\ls 1-\bls 1}\le \frac 1 t C\no{\ls t-\bls t}^{q/p}\iq{for} 0<t\le 1. 
}
In consequence, if  $0<t\le s\le 1,$ 
\eq{holdera}
{
\no{\ls 1-\bls 1}\le \frac 1 t 2^{q/p}C\no{\ls t-\bls s}^{q/p}.
}
\end{proposition}
\begin{remark}\label{condr} The unit length condition and \reff{param} imply that 
\eqn
{
\no{\ls0-\bls0}=\no{(\bls1-\bls0) -(\ls1-\ls0)+\ls1-\bls1}\le 2+R,
}
and then 
\eqn
{
\no{\ls t-\bls t}\le (1-t)\no{\ls0-\bls0}+t\no{\ls1-\bls1}\le 2+R. 
}
For this reason we relabel $2+R$ as simply ``$R$" and simply assume hereafter that
\eq{esti}
{
\no{\ls t-\bls t}\le R \iq{for} 0\le t\le 1. 
}
Further note that then 
\eq{rt}
{
\no{\ls t-\bls t}\le R^{1-q/p}\no{\ls t-\bls t}^{q/p},
}
as $q\le p.$ 
\end{remark}
\begin{proof}

 First we establish \reff{holder} for $t=1/2.$  Put
\eq{ncd}
{
e:=\ls1-\ls0, \ \be:=\bls1-\bls 0.
}
Noting that
\eqn
{
\ls1=\ls{1/2}+\frac 12 e, \  \ls0=\ls{1/2}-\frac 1 2e,\ \bls1=\bls{1/2}+\frac 12 \be, \  \bls0=\bls{1/2}-\frac 1 2\be,}
 the two sticks and equal length conditions are
\eqs{dat}
{
&1=\no{e}\le\no{\ls1-\bls0}=\no{\frac 1 2(e+\be)+\ls{1/2}-\bls{1/2}},\\
& 1=\no{\be}\le\no{\bls1-\ls0}=\no{\frac 1 2(e+\be)+\bls{1/2}-\ls{1/2}}.}
The desired estimate \reff{holder} for $t=1/2$ is of the form
\eq{des}
{
\no{\ls1-\bls1}=\no{\ls{1/2}-\bls{1/2}+\frac 1 2(\be-e)}\le C\no{\ls{1/2}-\bls{1/2}}^{q/p},
}  
where the meaning of $C$ varies according to need.

Thus there are only three vectors to be concerned about, $e, \be,$ and
\eq{q}
{
m:=\ls{1/2}-\bls{1/2};
}
the notation is a mnemonic for ``middle". 
In these terms, we want
\eq{ren}{
1=\no{e}\le \no{m+\frac 1 2(e+\be)},\ 
1=\no{\be}\le \no{-m+\frac 1 2(e+\be)},
}
to imply 
\eq{ref}
{
\no{m+\frac12(e-\be)}\le C\no{m}^{q/p}.
}
If we show, instead, that \reff{ren} implies
\eq{refa}
{
\no{e-\be}\le C\no{m}^{q/p}, 
}
with some other constant $C, $ then 
\eq{newc}
{
\no{m+\frac12(e-\be)}\le C\no{m}^{q/p}+\no m,
}
and Remark \ref{condr} takes us back to the form \reff{ref}.   For the moment, we will  obtain the bound \reff{refa} and leave the resulting \reff{newc} in ``raw" form.

From the two sticks condition \reff{ren} and the uniform smoothness assumption \reff{bsme}, we have 
\eqn
{
2\le \no{m+\frac 1 2(e+\be)} + \no{-m+\frac 1 2(e+\be)}\le \no{e+\be} +\frac{2^{q-1}B}{\no{e+\be}^{q-1}}\no{m}^q,
}
or
\eqns
{
2-\no{e+\be}\le \no{m+\frac 1 2(e+\be)} + \no{-m+\frac 1 2(e+\be)}-\no{e+\be}\le \frac{2^{q-1}B}{\no{e+\be}^{q-1}}\no{m}^q.
}
This estimate deteriorates when $\no{e+\be}$ is small. To handle this, we note again, as in Remark \ref{ysm}, that the intermediate term above is never more that $2\no m.$ Thus we consider cases as follows:
\eqs{fin}
{
2-\no{e+\be}\le \left\{\begin{split}&2^{q-1}B\no m^q\iq{if} \no{e+\be}\ge 1,\\ &2\no m\iq{if} \no{e+\be}\le 1. \end{split}\right.
}

 Combining \reff{bsce} and \reff{fin}, we find:
 \eq{combo}
 {
A\no{e-\be}^p \le \left\{\begin{split}&2^{q-1}B\no m^q\iq{if} \no{e+\be}\ge 1,\\ &2\no m\iq{if} \no{e+\be}\le 1. \end{split}\right.
 }
 Next note that if $\no{e+\be}\le 1,$ then
 \eqn
 {
 \no{e-\be}=\no{e+\be-2\be}\ge 2\no{\be}-\no{e+\be}\ge 2-1=1.
 }
 Therefore, in this case, we use $q\ge 1$ to find
 \eqn
 {
 1\le \no{e-\be}^p\le \frac{2}A\no m\implies \no{e-\be}^p\le \frac{2^q}{A^q}\no m^q.
 }
 Therefore, choosing $C(A,B,p,q)$ appropriately, \reff{combo} implies the  estimate
 \eq{ans}
 {
 \no{e-\be}\le C(A,B,p,q)\no m^{q/p}. 
 }
 Recalling what we were about,  we have established \reff{newc} with $C$ as above, or 
 \eq{oh}
 {
 \no{\ls 1-\bls 1}\le  C(A,B,p,q)\no {\ls{1/2}-\bls{1/2}}^{q/p}+\no {\ls{1/2}-\bls{1/2}}.
 }
 
 Next let us observe that if sticks $l^*, \tilde l$ have equal lengths $L\le 1,$ which is not necessarily 1, and satisfy the two sticks condition, we may apply \reff{oh} to $l^*/L, \tilde l/L$ (with the obvious meaning) to find
 \eqs{oho}
 {
 \no{l^* _1-\tilde l_1}&\le  C(A,B,p,q)L^{1-q/p}\no {l^*_{1/2}-\tilde l_{1/2}}^{q/p}+\no {l^*_{1/2}-\tilde l_{1/2}}
 \\ &\le  C(A,B,p,q)\no {l^*_{1/2}-\tilde l_{1/2}}^{q/p}+\no {l^*_{1/2}-\tilde l_{1/2}},
 }
 where we used $q\le p.$ 

To proceed, we next treat  the case $0<t\le1/2.$  With this assumption, we  note that
\eqn
{
\ls 1=\frac{1-t}{t} \ls {2t}-\frac {1-2t}t\ls t,\ \bls 1=\frac{1-t}{t} \bls {2t}-\frac {1-2t}t\bls t, 
}
and, from this, 
\eqn{
\ls 1-\bls 1= \frac{1-t}{t}( \ls {2t}-\bls {2t})-\frac {1-2t}t(\ls t-\bls t),
}
and then
\eq{scale}
{
\no{\ls 1-\bls 1}\le \frac 1 t\Lp(1-t)\no{\ls{2t}-\bls{2t}}+(1-2t)\no{\ls t-\bls t}\Rp.
}
Now we apply the estimate \reff{oho} to the pair of sticks $\tilde l=[\ls0,\ls{2t}], l^*=[\bls0,\bls{2t}],$ which have length $2t\le1$ and midpoints $\ls t, \bls t,$   to conclude that
\eq{new}
{
\no{\ls{2t}-\bls{2t}}\le C(A,B,p,q)\no{\ls t-\bls t}^{q/p}+\no{\ls t-\bls t}.
}
 Plugging this into \reff{scale} while using $t\le 1/2,$ we find
\eq{tleh}
{
\no{\ls 1-\bls 1}\le \frac 1 t\Lp C(A,B,p,q)\no{\ls t-\bls t}^{q/p}+2\no{\ls t-\bls t}\Rp.
 }

Finally, we assume that $1/2\le t< 1.$  This time we apply the estimate \reff{oh} to the pair of sticks $\tilde l=[\ls{2t-1},\ls1], l^*=[\bls{2t-1},\bls1]$ which have length $2t-1$ and midpoints $\ls t, \bls t,$ to conclude that
\eqn
{
\no{\ls 1-\bls 1}\le C(A,B,p,q)\no {\ls{t}-\bls{t}}^{q/p}+\no {\ls{t}-\bls{t}}.
}

The estimate \reff{tleh} and the estimate just above combine with Remark \ref{condr} to  establish \reff{holder} with a suitable $C,$  a process which causes $C$ to depend on $R$ as well as $A,B,p,q.$

As in the Euclidean case, \reff{holdera} holds if 
\eq{casea}
{
\no{\bls t-\bls s}\le\frac 12\no{\ls t-\bls t}
}
for then 
\eqn
{
\no{\ls t-\bls s}\ge \no{\ls t-\bls t}-\no{\bls t-\bls s}\ge \frac 12\no{\ls t-\bls t}. 
}
On the other hand, if \reff{casea} does not hold, then, by \reff{flipa}, 
\eqn
{
\no{\ls t-\bls s}\ge \no{\bls t-\bls s}\ge \frac 1 2\no{\ls t-\bls t}.
}
In both cases, we deduce \reff{holdera} from \reff{holder}.\end{proof}

\section{More General Norms:  Preliminaries}\label{mgn}

Since the remainder of this paper can be made entirely self contained and eminently accessible (as was the previous material) with little trouble, we will do so. Thus we review some standard facts and nomenclature.

We assume throughout that $\no{\cdot}$ is strictly convex. This amounts to the assumption that  if $x, y\not=0$ and 
\eqn
{
\no{x+y}=\no x+\no y,
}
then $x, y$ are ``positively parallel," i.e., $x=\alpha y$ for some $\alpha>0.$ Let us give this notion a formal definition, so as to make clear how we use the term  ``positively parallel".

\begin{definition}\label{dpara} Let $x, y\in\Rn.$ 
Then $y$ is positively parallel to $x$ if $y=\alpha x$ holds with $\alpha>0.$ 
\end{definition}

Note that ``positively parallel" is a symmetric relation. 

We also assume throughout that  $x\mapsto \no x$ is continuously differentiable on $\Rn\setminus \set{0}.$   The gradient of $\no{x}$ is denoted by $N(x).$ $N(x)$ is an exterior normal at $x$  to the ball of radius $\no x$ centered at the origin, with a certain normalization explained below.  We use $Dg$ to denote the gradient of $g:\Rn\ra \R,$ so 
\eq{dno}
{
N(x)=D\no x. 
}
Using the homogeneity of the norm, if $t>0,$ on the one hand
\eqn
{
\dt\no{t x}=\dt (t\no x)=\no x
}
and on the other
\eqn
{
\dt\no{t x}=\ip{x,N(tx)}
}
for $t>0.$
Hence
\eq{propj}
{
\no x=\ip{x,N(x)}.}
Here we use the notation \reff{dip}.  Somewhat redundantly, 
\eqn
{
D\no{tx}=tN(tx) \iq{and} D\no{tx}=Dt\no x=tD\no x=tN(x)
}
shows that $N(tx)=N(x)$ for $t>0.$  In the same way,  $\no{x}=\no{-x}$ implies that $N(-x)=-N(x).$  Finally, 
\eqn
{
\ip{y,N(x)}=\dt \no{x+ty}\Big|_{t=0} =\lim_{t\da 0}\frac{\no{x+ty}-\no x}t\le \lim_{t\da 0}\frac{\no{x}+t\no{y}-\no x}t=\no y. 
}

We have established the following properties of $N,$ which are used later without further comment: for $x\not=0,$ $t>0,$ 
\eqns
{
\ip{x,N(x)}=\no x, \ N(tx) = N(x),\ N(-x)=-N(x),\ \ip{y,N(x)}\le \no y. 
}

The strict convexity of $\no{\cdot}$ is reflected in $N$ in the following way. If $x, y\not=0$ and $N(x)=N(y),$ then $x$ and $y$ are positively parallel.  Indeed, the assumption implies that 
\eqn
{
\no{x}+\no{y}=\ip{x,N(x)}+\ip{y,N(y)}=\ip{x+y,N(x)}\le \no{x+y}.
}
By strict convexity, $\no x+\no y\le \no{x+y}$ implies  that $x$ and $y$ are positively parallel.

The converse also holds. For this, we note that $N(x)$ is  the unique vector $z$ such that
\eq{ud}
{
\ip{x,z}=\no x, \ip{y,z}\le \no y \iq{for} y\in\Rn.  
}
To see this, observe that \reff{ud} implies that
\eqns
{
\no{x+ty}=\no x+\no{x+ty}-\no x&= \no x+\no{x+ty}-\ip{x,z}
\\&\ge 
\no x+\ip{x+ty,z}-\ip{x,z}=\no x+t\ip{y,z}. 
}
Hence
\eqn
{
\dt \no{x+ty}\Big|_{t=0}=\ip{y,N(x)}\ge \ip{y,z}
}
for every $y.$ This entails $N(x)=z.$  Suppose that $x, y\not=0$ and 
\eqn
{
\no{x+y}=\ip{x+y,N(x+y)}=\ip{x,N(x+y)}+\ip{y,N(x+y)}=\no x+\no y. 
}
 By the preceding remark, $N(x+y)=N(x)=N(y).$ Hence $x$ and $y$ are positively parallel. 
 
\subsection{Geometric Convexity}\label{spc}
For $x\in \Rn\setminus\set{0}$ the function
\eq{es}
{
h(x,y):=\no{y}-\ip{y,N(x)}=\no{y}-\no x-\ip{y-x,N(x)} 
}
is the difference between $\no y$ and the linearization of $\no\cdot$ at $x$ evaluated at $y.$

\begin{remark} \label{props} The following  properties of $h$ will be used later, often without comment.  It is assumed that $x\not=0.$ The properties are:
\eqs{hprop}
{
{\rm (i)\ }& h(\alpha x,y)=h(x,y)\i{for} \alpha>0,
\\
{\rm (ii)\ }& h(- x,y)=h(x,-y),
\\
{\rm (iii)\ }& h(\alpha x,\alpha y)=\alpha h(x,y)\i{for} \alpha>0,
\\ 
{\rm (iv)\ }& h(-x,- y)=h(x,y),
\\
{\rm (v)\ }& y\mapsto h(x,y) \i{is convex,}
\\
{\rm (vi)\ }& h(x,\alpha x) =0 \i{for} 0\le \alpha,
\\
{\rm (vii)\ }& h(x,\alpha x) =2|\alpha|\no x \i{for} \alpha\le 0,
\\
{\rm (viii)\ }& h(x,y) =0 \i{iff}  y=\alpha x\i{for some} \alpha\ge 0.
}
\end{remark}
These relations follow directly  from properties of $N$ previously discussed and the definition of $h.$  Perhaps (viii) deserves comment.  Now $h(x,y)=0$ amounts to
\eqn
{
\no y=\ip{y,N(x)},
}
and we know that this implies $N(x)=N(y)$ if $y\not=0$ (see \reff{ud}).  Since $\no{\cdot}$ is strictly convex, this implies that $x$ and $y$ are positively parallel if $y\not=0.$  

\begin{definition}\label{dpc} We say that $\no{\cdot}$ is {\it geometrically convex} with constants $r, \Lambda$, where  
\eq{rl}
{
 0<r\iq{and} 2< \Lambda,
}
 provided that 
\eq{dpce}{
 \Lambda h(x,x+y)\le h(x,x+2y)\iq{for}x\not=0\iq{and} \no y\le r\no x.
}
If we merely say that ``$\no{\cdot}$ is geometrically convex", this means that it is geometrically convex with some constants $r, \Lambda$ satisfying \reff{rl}.\end{definition}

By homogeneity, \reff{dpce} holds iff it holds when $\no x=1.$ 

\subsection{Some Consequences of Geometric Convexity}\label{sc}
 We first notice that the range of $y$ for which an estimate of the form \reff{dpce} holds can be taken as large as desired. The lemma states that it can be doubled, and then, of course, it can be doubled again, etc. In this regard, notice that $\Lambda>2$ implies $3-2/\Lambda>2.$
\begin{lemma} \label{epc} Let \reff{dpce} hold. Then 
\eq{dpcee}
{
 \Lp 3-\frac 2\Lambda\Rp h(x,x+y)\le h(x,x+2y)\iq{for}x\not=0\iq{and} \no y\le 2r\no x.
}
That is, if $\no{\cdot}$ is geometrically convex with constants $r, \Lambda,$ then it is also geometrically convex with constants $2r$, $3-2/\Lambda.$ 
\end{lemma}
\begin{proof}
Let 
\eqn
{
g(y):=h(x,x+y). 
}
We use only the convexity of $g$ and 
\eq{cpc}
{
g(2y)\ge \Lambda g(y)
}
to conclude that 
\eq{cpd}
{
g(4y)\ge \Lp 3-\frac 2\Lambda \Rp g(2y). 
}
Via convexity and \reff{cpc},
\eqn
{
g(4y)-g(2y)\ge 2(g(2y)-g(y))\ge 2\Lp g(2y)-\frac 1\Lambda g(2y)\Rp.
}
Therefore 
\eqns
{
g(4y)= g(4y)-g(2y)&+g(2y)\ge 
\\
&2\Lp g(2y)-\frac 1\Lambda g(2y)\Rp+g(2y)=\Lp 3-\frac 2\Lambda \Rp g(2y).
}
\end{proof}

\begin{definition}
The {\it modulus of geometric convexity} of $\no{\cdot}$ at $x$ is
\eq{dmc}
 {
 \sigma(x,t):=\max_{\no y\le t}h(x,x+y).
 }
\end{definition}

\begin{remark}\label{rema} Let $\no{\cdot}$ be geometrically convex with constants $r, \Lambda.$  Let $y$ satisfy 
\eqn
{
\no y\le t\le r\no x,\ \sigma(x,t)=h(x,x+y).
}
Then
 \eqn {\sigma(x,2t)\ge h(x,x+2y)\ge \Lambda h(x,x+y)= \Lambda \sigma(x,t),
 }
 so
$
\sigma(x,2t)\ge \Lambda \sigma(x,t).
$
 \end{remark}
 
 \begin{lemma}\label{remb}  Let $x\not=0,$ $t>0,$ and $y$ satisfy $\no y\le t$ and 
 \eqn
 {
 \sigma(x,t)=h(x,x+y). 
 }
 Then $\no y=t.$ Moreover, if  $\ip{w-(x+y),N(y)}\ge 0, $ then
 \eq{hs}
 {
 h(x,w)\ge h(x,x+y)=\sigma(x,t). 
 }
 That is, $w=x+y$ minimizes $h(x,w)$ for $w$ in the half space exterior to the ball $B_t(x)$ at $x+y$ with interior normal $N(y).$
  
  In consequence, 
  \eq{axy}
  {
  -t\le \ip{x,N(y)}\le 0. 
  }
 \end{lemma}
 
 \begin{proof}
 Note that the gradient of 
 \eqn
 {
y\mapsto h(x,x+y)=\no{x+y}-\ip{x+y,N(x)}
 }
 is $N(x+y)-N(x).$  This vanishes only if $x+y$ is positively parallel to $x,$ and then $h(x,x+y)=0.$  Hence it must be that $\no y=t.$ 
Using Lagrange multipliers and the assumptions on $y,$  there is an $\alpha> 0$ such that 
 \eq{nc}
 {
 N(x+y)-N(x)=\alpha N(y). 
 }
 Hence $\ip{w-(x+y),N(y)}\ge 0$ implies that
 \eqn
 {
 \ip{w-(x+y),N(x+y)-N(x)}\ge 0.
 } 
 Therefore
 \eqs{zp}
 {
 h(x,w)=\no{w}& -\ip{w,N(x)}\\& \ge \ip{w,N(x+y)}-\ip{w,N(x)}\\ 
 &=\no{x+y}+\ip{w-(x+y),N(x+y)-N(x)}-\ip{x+y,N(x)}\\&\ge\no{x+y}-\ip{x+y,N(x)}=h(x,x+y).
 }
 
 To establish \reff{axy}, we note that by
 \eqn
 {
 0=h(x,x+sx)<h(x,x+y) \iq{for} -1\le s
 }
 and the claim already proved, we have
 \eqn
 {
 \ip{(x+sx)-(x+y),N(y)}=s\ip{x,N(y)}-\ip{y,N(y)}=s\ip{x,N(y)} -t \le 0. 
 }
 Taking $s=-1$ establishes the left most inequality of \reff{axy}, while letting $s\ra\infty$ proves the right most inequality of \reff{axy}. 
  \end{proof}
 \begin{remark} If $g$ is a convex function, then one has 
 \eqn
 {
 g(w)\ge g(z)+\ip{w-z,Dg(z)}.
 }
 In particular, $z$ minimizes $g(w)$ over the half space 
$
 \ip{w-z,Dg(z)}\ge 0. 
 $
 We are really using this remark above, somewhat hidden. 
 \end{remark}
 \begin{example}\label{mch}
We provide an example, using  the Euclidean norm, which is written $\noe{\cdot}$ as before. This example, which is the only one we have computed, is not used later in the text. However, it does offer some insight. In this regard, it would be interesting to know what the set of maximizing $y'$s can look like in other cases, for example, the cases $\no{\cdot}=\nop{\cdot}.$ 

  Let us compute $\sigma(x,t)$ for $\noe x=1, 0<t\le 1.$  One has, in this case, 
\eq{hilj}
{
N(x)=\frac{x}{\noe x}.
}

Let $y$  maximizes $h(x,x+y)$ subject to $\noe y=t.$
By \reff{nc}, we know such a maximizing point satisfies
\eq{nch}
{
\frac{x+y}{\noe{x+y}}-x=\alpha y
}
for some $\alpha> 0.$ If $\alpha\noe{x+y}\not=1,$ this may be solved for $y:$
\eqn
{
y=\frac {1-\noe{x+y}}{\alpha\noe{x+y}-1}x=\frac  {1-\noe{x+y}}{\alpha(\noe{x+y}-1)+\alpha-1}x.
}
If the coefficient of $x$ on the right is positive, we know that $h(x,x+y)=0$ (Remark \ref{props}). Hence it must be negative. Since $\noe y=t,$ we then have  $y=-tx,$ or $x+y=(1-t)x,$ which still implies $h(x,x+y)=0,$ as $t\le1.$ 

Therefore $\noe{x+y}\alpha=1,$  and then \reff{nch} merely states that 
\eqn
{
\frac 1{\noe{x+y}} x-x=0
}
or $\noe{x+y}=1.$  On the other hand, this implies that
\eq{cip}
{
\noe x^2+\noe y^2+2\ip{y,x}=1+t^2+2\ip{y,x}=1, \iq{or} \ip{y,x}= -\frac{t^2}2.
}
Therefore
\eqn
{
h(x,x+y)=\noe{x+y}-1-\ip{y,x}=1-1+\frac{ t^2}2=\frac{t^2}2
}
and the maximizing $y's$ are just the points of the form $y=z-x$ where $\noe z=1, \noe{z-x}=t.$ 
\end{example}

The next lemma is crucial later. The property it asserts we call ``duality."

 \begin{lemma}\label{theta} Let $\no{\cdot}$ be geometrically convex  with constants $r, \Lambda.$ Then, for $\no {y}\le r\no x, $
 \eqn
{
 h(x,x+2y)\le \frac\Lambda{\Lambda-2}h(x+2y,x).
}
Equivalently, if $\no{z-x}\le 2r\no x,$ then
 \eqn
{
 h(x,z)\le \frac\Lambda{\Lambda-2}h(z,x).
}

\end{lemma}
\begin{proof}
By assumption,  
\eqns
{
\Lambda h(x,x+y)=\Lambda(\no{x+ y}&-\ip{x+y ,N(x)})\le  h(x,x+2y)\\
&=\no{x+2y}-\ip{x+2y,N(x)});
}
therefore
\eqns{
\Lambda \no{x+ y}&\le  h(x,x+2y)+\Lambda\ip{x+ y,N(x)}
\\&= \frac \Lambda2(\no{x+2y}-\ip{x+2y,N(x)})+\Lambda\ip{x+y,N(x)}-\Lp \frac \Lambda2-1\Rp h(x,x+2y)\\&=
\frac \Lambda2(\no{x+2y} +\ip{x,N(x)})-\Lp \frac \Lambda2-1\Rp  h(x,x+2y)\\
&=\frac \Lambda2(\no{x+2y}+\no x)-\Lp \frac \Lambda2-1\Rp h(x,x+2y).
}
Hence
\eqns{0&\le 
\Lambda(\big\|x+ y\big\|-\ip{x+ y,N(x+2y)})
\\
&\le\frac \Lambda2(\no{x+2y}+\no x)-\Lp \frac \Lambda2-1\Rp h(x,x+2y)-\Lambda\ip{x+ y,N(x+2y)}\\
&=\frac\Lambda2(\no{x+2y}+\no{x})-\frac\Lambda2\ip{x+ 2y,N(x+2y)}-\frac\Lambda2\ip{x,N(x+2y)}-\Lp \frac \Lambda2-1\Rp h(x,x+2y)
\\&=\frac\Lambda2(\no{x}-\ip{x,N(x+2y)})
 -\Lp\frac\Lambda2-1\Rp h(x,x+2y)
 \\&=
\frac \Lambda2 h(x+2y,x)-\Lp\frac\Lambda2-1\Rp h(x,x+2y).
}
It follows that
\eqn
{
 h(x,x+2y)\le \frac\Lambda{\Lambda -2}h(x+2y,x). 
}
The final assertion of the lemma results from putting $z=x+2y.$
\end{proof}
\begin{remark}\label{losc} We did not require $x+2y\not=0$ above, while several expressions above are undefined if this does not hold. The conclusion is still correct  if $x+2y=0,$ provided that we define $h(0,z)=0$ for all $z.$  This is the greatest lower-semicontinuous extension of $h$ to cases in which its first argument is 0.
\end{remark}
\section{The  Problem of Two Sticks and Geometric Convexity}\label{tspc}

In this section we assume, for simplicity, that  $\no{\cdot}$ is geometrically convex with  constants $1, \Lambda.$ In this regard, recall Lemma \ref{epc}.  
The next result provides a basic restriction on $\ls1-\bls1$ when $l, \bl$ satisfy the two sticks condition, the equal length condition with $L=1$ (a normalization) and meet a common small ball.  The nature of the theorem is perhaps not  transparent.  We forge ahead and state it directly and then offer some explanatory remarks.  

In the following statement,  the ``directions" of the sticks are denoted by  the unit vectors
\eq{ds}
{
e=\ls 1-\ls 0,\  \be=\bls 1-\bls 0.
}

\begin{theorem} \label{strip} Let $\no{\cdot}$ be geometrically convex with constants $1, \Lambda.$ Let $\z, \bl$ satisfy the two sticks and equal length conditions  with $L=1.$  Assume that $0<\delta<1/4, \xs 0\in\Rn,$ and
\eq{near}
{
\z\cap \bar B_{\delta}(\xs 0)\not=\emptyset,\   \bl\cap \bar B_{\delta}(\xs 0)\not=\emptyset.
}
Assume, moreover, that $\rho>3\delta,$
\eq{notinb}
{
\ls1, \ls0\notin \bar B_{\rho} (\xs 0), 
}
\eq{kappa}
{
\kappa\ge\frac 4{\rho-3\delta},
}
 and 
\eq{bmax}
{
\sigma(e,\kappa\delta)\le \sigma(\be,\kappa\delta). 
}
Then 
\eq{promise}{h(\be,\bls1-\ls0)+h(\be,\ls1-\bls0)\le \frac\Lambda{\Lambda-2}\sigma(\be,\kappa\delta).}
\end{theorem}

Here are some explanatory remarks about the statement and the proof to follow.  
First,  it follows from \reff{promise} that 
\eq{kp}
{
h(\be,\ls1-\bls0)=h(\be,\bls1-\bls0+\ls 1-\bls 1)=h(\be,\be+\ls1-\bls 1)\le \frac{\Lambda}{\Lambda -2}\sigma(\be,\kappa\delta).  
}
This is a restriction on where $\ls1-\bls1$ can lie.  In 
Corollary \ref{cstrip} below, it is parlayed  into forcing $\ls1-\bls1$ to lie between parallel planes which are separated by a width comparable to $\delta.$  

The statement involves  the somewhat mysterious  condition \reff{bmax}.   We have in mind, for use in \cite{cc},  not only a pair of sticks, but a collection of them which pairwise  satisfy the two sticks condition and all of which meet a small ball $\bar B_{\delta}(\xs0),$ where $\xs0$ is well away from the endpoints of the sticks. From this collection, we will choose a ``special stick." Here $\bl$ is the special stick, that is, the stick  for which $\sigma(e,\kappa\delta)$ is maximal, corresponding to \reff{bmax}.   Then \reff{promise} holds valid for the other sticks in the collection which also satisfy \reff{notinb}.

As regards the proofs, there is the  ``auxiliary" stick, $[\bls 0, \ls1].$  This is ``kinked" to the point $\blst$ in $\bl,$ as in \reff{wie}.  This kinking, as in \reff{wie}, produces a length gain (the term $\no{\ls1-\blst}$ in \reff{wie} vs $\no{\ls 1-\lst}$, as estimated in \reff{et}) which helps in competition with the strictness in the triangle inequality codified by the triangle equality explained below. In this kinking process, $\blst$ is the point ``kinked to," while $\lst$ is loosely thought of as the point ``kinked from."  A good point to kink from will satisfy the second relation of  \reff{lst}; this is so that \reff{et} holds.   This process is repeated with the second auxiliary stick, $[\ls0,\bls 1]$, kinking from $\blst$ to $\lst,$ and for all this  to end up consistent with the two sticks condition, the conclusion of the theorem must hold.

\proofof{Theorem \ref{strip}.}

By assumption, there exists a point 
\eq{perp}
{
\blst\in \bl\cap \bar B_{\delta}(\xs0).
}
We seek a point
\eq{lst}
{
\lst\in \z\cap \bar B_{3\delta}(\xs0)\iq{for which}\ip{\lst-\blst,N(e)}=0.
}
Suppose that $\li\in \bar B_{\delta}(\xs0)\cap l;$ then 
 \eq{nested}
 { 
 \bar B_{\delta}(\xs0)\subseteq \bar B_{2\delta}(\li)\subseteq \bar B_{3\delta}(\xs 0).
  }
Recalling \reff{notinb}, $\rho>3\delta,$  \reff{nested}, and $w\in \z,$ we may choose $0<t_1< t_2<1$ to be the values of $t$ at which $\ls t$ enters and leaves  $\bar B_{2\delta}(\li).$   Then 
\eqn
{
2\delta=\no{\ls{t_2}-\li}=\ip{\ls{t_2}-\li,N(e)} \i{and} -2\delta =-\no{\ls{t_1}-\li}=\ip{\ls{t_1}-\li,N(e)}. 
}
Hence
\eqns
{
\ip{\ls{t_2}-\blst,N(e)}& =\ip{\ls{t_2}-\li,N(e)}+\ip{\li-\blst,N(e)}
\\&
\ge 2 \delta -\no{\blst-\li}\ge 0;
}
the last relation is due to $\blst, w\in \bar B_\delta(\xs 0).$ 
Similarly, $\ip{\ls{t_1}-\blst,N(e)}\le 0.$ Thus there exists $t\in[t_1,t_2]$ such that \reff{lst} holds with $\lst=\ls t.$ Moreover, by $\lst, \blst\in \bar B_{2\delta}(w),$ 
\eq{nest}
{
\no{\lst-\blst}\le 4\delta. 
}
We also note that, via \reff{notinb}, 
\eqn
{
\no{\ls 1-\lst}\ge \no{\ls1 -\xs 0}-\no{\xs 0-\lst}\ge \rho-3\delta. 
}
Treating $\lst-\ls 0$ similarly and using \reff{kappa}, we find
\eq{frac}
{
\frac 4\kappa\le \min(\no{\ls1-\lst}, \no{\ls0-\lst}).
}

We next note two general identities which will play a role.  The first is a rewrite of the definition of $h,$ 
\eq{iha}
{
\no y=\no x+h(x,y)+\ip{y-x,N(x)}.
}
The second, which we call the ``triangle equality",  is
\eq{ihb}
{
\no{x+y}=\no x+\no y-h(x+y,x)-h(x+y,y).
}
This also follows immediately from the definition of $h,$ or, as we prefer,  
\eqns
{
\no{x+y}&=\ip{x+y,N(x+y)}=\ip{x,N(x+y)}+\ip{y,N(x+y)}
\\&=\no x-(\no x-\ip{x,N(x+y)})+\no y-(\no y-\ip{y,N(x+y)})
\\&
=\no x+\no y-(h(x+y,x)+h(x+y,y)).
}

Using the triangle equality \reff{ihb} we have
\eqs{wie}
{
\no {\ls1-\bls0}&=\no {\ls1-\blst +\blst-\bls0}
\\&=\no{\ls1-\blst}+\no{\blst-\bls0}-h(\ls1-\bls0,\ls1-\blst)
-h(\ls1-\bls0,\blst-\bls0).
}
We will combine this with the following consequence of \reff{iha}. Put $y=\ls 1-\blst$ and $x=\ls1-\lst$ in \reff{iha} and use \reff{lst}, \reff{nest} and \reff{frac} to find
\eqs{et}
{
\no{\ls1-\blst}&=\no{\ls 1-\lst}+h(\ls1-\lst,\ls1-\blst)+\ip{\lst-\blst, N(\ls1-\lst)}
\\& =\no{\ls 1-\lst}+h(\ls1-\lst,\ls1-\lst+\lst-\blst)
\\& \le \no{\ls 1-\lst}+\sigma(\ls1-\lst,\no{\lst-\blst}).
\\& =\no{\ls 1-\lst}+\no{\ls1-\lst}\sigma\Lp \frac{\ls1-\lst}{\no{\ls1-\lst}},\frac{\no{\lst-\blst}}{\no{\ls1-\lst}}\Rp
\\& \le\no{\ls 1-\lst}+\no{\ls1-\lst}\sigma\Lp e,\kappa\delta\Rp.
}

Using the estimate \reff{et} in \reff{wie} results in 
\eqs{half}
{
\no {\ls1-\bls0}&\le\no{\ls 1-\lst}+\no{\ls1-\lst}\sigma\Lp e,\kappa\delta\Rp+
\\&\no{\blst-\bls0}-h(\ls1-\bls0,\ls1-\blst)
-h(\ls1-\bls0,\blst-\bls0).
}
At this point we will drop the nonpositive term $-h(\ls1-\bls0,\ls1-\blst)$ from the right of \reff{half} and use duality (Lemma \ref{theta}) to replace $h(\ls1-\bls0,\blst-\bls0)$ by 
\eqn
{
\frac{\Lambda-2}\Lambda h(\blst-\bls0,\ls1-\bls0).
}
Recall that we are assuming geometrical convexity with constants 1, $\Lambda.$  The estimate needed to justify this last step is therefore, according to Lemma \ref{theta}, 
\eq{need}
{
\no{\ls1-\bls0-(\blst-\bls0)}=\no{\ls1-\blst}\le 2\no{\ls 1-\bls0}.
}
Now, using \reff{nest} and $\delta<1/4, $
\eqn
{
\no{\ls1-\blst}\le \no{\ls1-\lst}+\no{\lst-\blst}\le 1+4\delta<2,
}
while $1\le \no{\ls1 -\bls0}.$
Therefore  \reff{need} holds.  The result of these machinations is:
\eqs{halfa}
{
\no {\ls1-\bls0}&\le\no{\ls 1-\lst}+\no{\ls1-\lst}\sigma\Lp e,\kappa\delta\Rp+
\no{\blst-\bls0}
-\frac{\Lambda-2}\Lambda h(\blst-\bls0,\ls1-\bls0)
\\&=\no{\ls 1-\lst}+\no{\ls1-\lst}\sigma\Lp e,\kappa\delta\Rp+
\no{\blst-\bls0}
-\frac{\Lambda-2}\Lambda h(\be,\ls1-\bls0).
}

We run analogous estimates again: on the one hand
\eqs{halfd}
{
\no{\bls1-\ls0}=\no{\bls 1-\blst}+\no{\blst-\ls 0}-h(\bls 1-\ls0,\bls1-\blst)-h(\bls 1-\ls 0,\blst-\ls0),
}
and on the other
\eqs{halfe}
{
\no{\blst -\ls 0}&=\no{\lst-\ls0}+h(\lst -\ls0,\blst-\ls0)-\ip{\lst-\blst,N(\lst-\ls0)}
\\
&=\no{\lst-\ls0}+h(\lst -\ls0,\lst-\ls0 +\blst-\lst)
\\&\le \no{\lst-\ls0}+\sigma(\lst-\ls0,\no{\blst-\lst})
\\ 
&\le\no{\lst-\ls0}+\no{\lst-\ls0}\sigma\Lp\frac{\lst-\ls0}{\no{\lst-\ls0}},\kappa\delta\Rp 
\\&=\no{\lst-\ls0}+\no{\lst-\ls0}\sigma\Lp e,\kappa\delta\Rp. 
}
Combining \reff{halfd} and \reff{halfe} and playing the same game as before results in
\eq{halff}
{
\no{\bls 1-\ls0}\le \no{\bls1-\blst}+\no{\lst-\ls0}+\no{\lst-\ls0}\sigma\Lp e,\kappa\delta\Rp-\frac{\Lambda-2}{\Lambda}h(\be,\bls 1-\ls 0). 
}
  Adding \reff{halfa}, \reff{halff} and using 
\eqns
{
&2\le \no {\bls0-\ls1}+\no {\ls1-\bls0},\\
&2=\no{\ls1-\lst}+\no{\lst-\ls0}+\no{\bls0-\blst}+\no{\blst-\bls0},
}
we arrive in the promised land  
\eqn
{
\frac{\Lambda-2}\Lambda(h(\be,\bls1-\ls0)+h(\be,\ls1-\bls0))\le \sigma(e,\kappa\delta).
}
The final assumption of the theorem now yields its assertion. 
\hfill $\square$

In the next result, we also assume that $\no{\cdot}$ is ``balanced" in the the following sense.  
\begin{definition} The norm $\no{\cdot}$ is {\it balanced} if there are constants $R>0, K\ge 1$ for which 
\eq{bal}
{
h(x,x+y)\le Kh(x,x-y)\iq{for} \no y\le R\no x. 
} 
\end{definition}
As with geometric convexity, by homogeneity, \reff{bal} holds in general if it holds with $\no x=1.$ This condition is explored further, together with geometric convexity, in Section \ref{vgc}. We remark that the assumption $K\ge 1$ is redundant in that it is implied by \reff{bal}. 

\begin{corollary}\label{cstrip} Let the assumptions of Theorem \ref{strip} be satisfied.  In addition, assume that \reff{bal} holds. Assume further that 
\eq{eta}
{
\no{\ls1-\bls1}\le R\iq{and} \frac{K\Lambda^2}{\Lambda-2}\kappa\delta\le 1. 
}  Let $\by$ satisfy
\eq{by}
{
\no {\by}=\frac{K\Lambda^2}{\Lambda-2}\kappa\delta, \ h(\be,\be+\by)=\sigma\Lp \be, \frac{K\Lambda^2}{\Lambda-2}\kappa\delta\Rp.
}
Then 
\eq{strip!}{
-\frac{K\Lambda^2}{\Lambda-2}\kappa\delta\le \ip{\ls 1-\bls 1,N(\by)}\le \frac{K\Lambda^2}{\Lambda-2}\kappa\delta.
}
Moreover, 
\eq{axya}
{
-\frac{K\Lambda^2}{\Lambda-2}\kappa\delta\le \ip{\be, N(\by)}\le 0.
}
\end{corollary}
\begin{proof} Let $N$ be the least integer for which 
\eq{dn}
{
\frac{K\Lambda}{\Lambda-2}<\Lambda^N. 
}
Then 
\eq{edn}
{
\Lambda^{N-1}\le \frac{K\Lambda}{\Lambda-2}\implies \Lambda^N\le \frac{K\Lambda^2}{\Lambda-2}.
}
Moreover,
\eq{tdn}
{
2^N<\Lambda^N\le \frac{K\Lambda^2}{\Lambda-2},
}
and so, by Remark \ref{rema}, 
\eq{ndn}
{
\frac{K\Lambda}{\Lambda -2} \sigma(\be,\kappa\delta)<\Lambda^N\sigma(\be,\kappa\delta)\le \sigma\Lp\be, 2^N\kappa\delta\Rp\le \sigma\Lp\be, \frac{K\Lambda^2}{\Lambda-2}\kappa\delta\Rp.
}

Let $\by$ be as in \reff{by}. 
Then, by Lemma \ref{remb} and \reff{ndn}
\eq{hmin}
{
\frac{K\Lambda}{\Lambda -2} \sigma(\be,\kappa\delta)<\sigma\Lp\be, \frac{K\Lambda^2}{\Lambda-2}\kappa\delta\Rp\le h(\be, w)\iq{if} \ip{w-(\be+\by),N(\by)}\ge 0. 
}
Recalling that $1\le K,$ according to  \reff{promise} of Theorem \ref{strip}, we must therefore have 
\eqn
{
\ip{\ls1-\bls 0-(\be+\by),N(\by)}=\ip{\ls1-\bls1-\by,N(\by)}<0,
}
which amounts to the rightmost inequality of \reff{strip!}, as 
\eqn
{
\ip{\by,N(\by)}=\no{\by}=\frac{K\Lambda^2}{\Lambda-2}\kappa\delta.
}
For this all to be valid, we need $2^k\kappa\delta,$ $k=1,2,\ldots,N,$  to remain in the range where Remark \ref{rema} applies.  With the observations above, this is exactly the role of the second condition in \reff{eta}. 

On the other hand, suppose that
\eq{sps}
{
\ip{\ls 1-\bls0-(\be-\by),N(\by)}\le 0.
}
Then 
\eqn
{
\ip{-\ls 1+\bls0+2\be -(\be+\by),N(\by)}\ge 0,
}
so 
\eq{dog}
{
h(\be, -\ls 1+\bls0+2\be)\ge \sigma\Lp\be, \frac{K\Lambda^2}{\Lambda-2}\kappa\delta\Rp.
}
But, by \reff{bal}, \reff{kp}, 
\eqs{doga}
{
h(\be, -\ls 1+\bls0+2\be)&=h(\be,\be+( -\ls 1+\bls0+\be))
\\&
\le Kh(\be, \be-( -\ls 1+\bls0+\be))=Kh(\be,\ls 1-\bls0)\le\frac{K\Lambda}{\Lambda -2} \sigma(\be,\kappa\delta). 
}
provided that
\eqn
{
\no{\bls0-\ls1+\be}=\no{\bls1-\ls1}\le R\no{\be}=R,
}
a condition we assumed in \reff{eta}.
The relations \reff{dog}, \reff{doga} are inconsistent in view of \reff{ndn}, so \reff{sps} does not hold.  That is,
\eqn
{
\ip{\ls 1-\bls0-(\be-\by),N(\by)}=\ip{\ls 1-\bls 1+\by,N(\by)}\ge 0,
}
which is the leftmost inequality of \reff{strip!}. 

 Finally, \reff{axya} is merely an incarnation of \reff{axy}; it is restated so as to have all essential information in one place. 

 \end{proof}
\section{\label{vgc} Verifying Geometric Convexity, etc. }

The main goal of this section is to prove that the norms $\nop{\cdot}$ satisfy all the conditions used in Section  \ref{tspc}.  This verification was no small task.  We succeeded in verifying these estimates with a direct approach only for $3/2<p<\infty,$ and this led us eventually to the  auxiliary concepts ``in the tangent plane" employed in this section, from which full estimates can be then derived.  

\subsection{\label{ntp} Notions ``in the Tangent Plane"}

\begin{definition}\label{dpctp} The norm $\no{\cdot}$ is {\it geometrically convex in the tangent plane} with constants $\Lambda, r,$  provided that $\Lambda>2, r>0,$ and 
\eq{pctp}
{
h(x,x+2y)\ge \Lambda h(x,x+y) 
}
whenever
\eq{pt}
{
\no y\le r\no x\iq{and}\ip{y,N(x)}=0. 
}
\end{definition}

The second condition of \reff{pt} is that $y$ is a tangent direction to the sphere through $x$ at $x.$ In the proceedings we will use, without further comment, that if $\ip{y,N(x)}=0,$ then 
\eqn
{
h(x,x+y)=\no{x+y}-\ip{x+y,N(x)}=\no{x+y}-\no x. 
}
\begin{definition}\label{doubling}The norm $\no{\cdot}$   is {\it doubling} with constants $T, r$ if  
\eq{fpc}
{
h(x,x+2y)\le Th(x,x+y) \iq{for}\no y\le r\no x. 
}
The norm $\no{\cdot}$   is {\it doubling in the tangent plane} with constants $T, r$ if  \reff{fpc} holds provided also that $\ip{y,N(x)}=0.$  
\end{definition}
\begin{definition}\label{baltp}The norm $\no{\cdot}$   is {\it balanced in the tangent plane} with constants $r, K>0$ if 
\eq{btp}
{
h(x,x+y)\le Kh(x,x-y)
}
whenever \reff{pt} holds. 
\end{definition}

If we say that $\no{\cdot}$ is geometrically convex in the tangent plane, this means it is geometrically convex in the tangent plane with some constants $\Lambda, r;$ likewise for the various doubling and balanced conditions.  However, the next result shows that the parameter $r$ can be dispensed with in the tangential doubling and balanced conditions. 

\begin{lemma}\label{uni} Let $\no{\cdot}$ be doubling in the tangent plane. Then there exists a constant $T$ such that 
\eq{gdub}
{
h(x,x+2y)\le T h(x, x+y) \iq{for} x\not=0, \ip{y,N(x)}=0. 
}
Similarly, let $\no{\cdot}$ be balanced in the tangent plane.  Then there exists a constant $K$ such that 
\eq{bdub}
{
h(x,x+y)\le Kh(x, x-y) \iq{for}x\not=0, \ip{y,N(x)}=0. 
}
\end{lemma} 
\begin{proof}
We establish \reff{gdub}. First notice that, by homogeneity,  if \reff{gdub} holds for $\no x=1,$ then it holds for all $x\not=0.$ Thus we assume that $\no x=1.$ If there is no such $T,$ then there exists sequences $\xs j, \ys j$ such that
\eqn
{
\no{\xs j}=1,\ \ys j\not=0,\ \ip{\ys j, N(\xs j)}=0\iq{for}j=1,2,\cdots ,
}
for which
\eqn
{
h(\xs j,\xs j+2\ys j)> j h(\xs j, \xs j+\ys j). 
}
We may assume that $\xs j\ra x$ for some unit vector $x$ and one of 
\eqn
{
{\rm (i)}\ \ys j\ra 0,\ {\rm (ii)}\ \no{\ys j}\ra \infty,\ {\rm (iii)}\ \ys j\ra y\not=0,
}
holds. Case (i) cannot occur, since we assumed that $\no{\cdot}$ is doubling in the tangent plane.  Case (ii) cannot occur, for 
\eqn
{
h(\xs j,\xs j+2\ys j)\le 2\no{\ys j},\    h(\xs j, \xs j+\ys j)\ge \no{\ys j}-2. 
}
Case (iii) cannot occur, for then
\eqn
{
h(\xs j, \xs j+2\ys j)\ra h(x,x+2y),\ h(\xs j, \xs j+\ys j)\ra h(x,x+y)\not=0.
}
The assertion that $h(x,x+y)\not=0$ at the end above holds since $y\not=0$ and, clearly, $\ip{y,N(x)}=0,$ which imply that $x+y$ is not positively parallel to $x.$ The proof of \reff{bdub} runs the same way. 
\end{proof}

\subsection{\label{tptf} From Tangent Plane Estimates to Full Estimates}
The proofs of the results stated in the next theorem contain estimates which  are referred to elsewhere, but are not recorded in the theorem itself. 

\begin{theorem}\label{dtptp} Let $\no{\cdot}$ be doubling in the tangent plane. Then $\no{\cdot}$ is doubling. Moreover, if $\no{\cdot}$ is also geometrically convex in the tangent plane, then $\no{\cdot}$ is geometrically convex.  Further, if $\no{\cdot}$ is also balanced in the tangent plane, then it is balanced.  
\end{theorem}

We prepare another simple lemma.  Given a unit vector $x$,  because $\ip{x,N(x)}=1\not=0,$  we can decompose an arbitrary vector $y$ into the sum of a scalar multiple of $x$ and a vector perpendicular to $N(x).$ We record this, with some more notational detail we will use.   
\begin{lemma}\label{decomp}
Let $\no x=1$ and $y\in\Rn.$  Then  there exists $\varepsilon, \alpha\in\R$ and $x^\perp\in \Rn$ with the properties
\eq{edc}
{
y=\alpha x+\varepsilon x^\perp,\ \varepsilon\ge 0,\  \no{x^\perp}=1, \ \ip{x^\perp,N(x)}=0.
}
Moreover, $\alpha, \varepsilon,$ are unique and $ x^\perp$ is unique if $y$ is not a multiple of $x.$ 
\end{lemma}
\begin{proof}
If \reff{edc} holds,  $\alpha$ may be computed by
\eqn
{
\ip{y,N(x)}=\alpha\ip{x,N(x)}+\varepsilon \ip{x^\perp, N(x)}=\alpha;}
thus 
\eq{dela}
{
\alpha=\ip{y,N(x)}. 
}
With this $\alpha,$ $y-\alpha x$ is orthogonal to $N(x).$ If $y=\alpha x,$ then $\varepsilon=0$ and we may choose $x^\perp$ to be any unit vector orthogonal to $N(x)$.  If $y\not=\alpha x,$ then $\varepsilon\ge 0$ and \reff{edc} imply
\eq{xperp}
{
x^\perp=\frac{y-\alpha x}{\no{y-\alpha x}},\ \varepsilon =\no{y-\alpha x},
} 
and we have our decomposition, whose uniqueness is evident. 
\end{proof}

\proofof{Theorem \ref{dtptp}.}
We assume throughout that $\no{\cdot}$ is doubling in the tangent plane. Then we invoke Lemma \ref{uni} to assume that \reff{gdub} holds. 

 To begin, we  
assume that $\no{\cdot}$ is also geometrically convex in the tangent plane with  constants $\Lambda, r$ and show that then $\no{\cdot}$ is geometrically convex  with constants established during the proof.  These arguments set the format for showing that $\no{\cdot}$ is doubling, and, if it is balanced in the tangent plane, then it is balanced.  This last we leave to the reader, the pattern having been well established by that point.  

 We may assume that $\no x=1.$   We use the coordinates of \reff{edc} throughout and assume that 
\eq{bdel}
{
|\alpha|\le \kappa\le \frac 14.
}
where $\kappa>0$ will be further restricted later.  This guarantees that $0<1/2\le 1+\alpha, 1+ 2\alpha.$

Observe that our assumptions imply 
\eqs{dtpb}
{
h(x,x+2y)&= h(x, x+2(\alpha x+\varepsilon x^\perp))
\\&
=h((1+2\alpha) x, (1+2\alpha) x+2\varepsilon x^\perp)
\\&
=\no{(1+2\alpha) x+2\varepsilon x^\perp}-(1+2\alpha)\no x
\\&
\ge \Lambda( \no{(1+2\alpha) x+\varepsilon x^\perp}-(1+2\alpha)\no x)
\\&
=\Lambda(1+2\alpha)\Lp \no{ x+\frac{\varepsilon}{1+2\alpha} x^\perp}-\no x\Rp
}
provided that 
\eq{pth}
{
\varepsilon\le r\no{(1+2\alpha)x}=(1+2\alpha)r. 
}
We seek to bound $h(x,x+2y)$ below by a multiple greater than $2$ of $h(x,x+y).$ Now, as above,  
\eqs{wa}
{
h(x,x+y)&= \no{(1+\alpha) x+\varepsilon x^\perp}-(1+\alpha)\no x
\\&=(1+\alpha)\Lp \no{x+\frac{\varepsilon}{1+\alpha}x^\perp}-\no x\Rp.
}
If $\varepsilon=0,$ then $h(x,x+y)=h(x,x+2y)=0. $ We assume, therefore, that $\varepsilon>0,$  and then $h(x,x+y), h(x,x+2y)>0$. 
Define 
\eq{dg}
{
g(t)=\frac{\no{x+t\varepsilon x^\perp}-\no x}{\no{x+\varepsilon x^\perp}-\no x}.
}
Then 
\eq{pg}
{
g:[0,\infty)\ra [0,\infty)\i{is convex,} \lim_{t\downarrow 0}\frac{g(t)}t=0,\  g(1)=1. }

Moreover, by \reff{gdub}, 
\eq{ub}
{
g(4)\le Tg(2)\le T^2 g(1)=T^2.
}
As  $g$ is  convex and satisfies \reff{pg}, \reff{ub}, it is Lipschitz continuous on  [0,2].  Note that \reff{bdel} implies 
\eq{bda}
{
 0\le \frac1{1+2\alpha}, \frac1{1+\alpha}\le \frac 1{1-2\kappa}\le 2.
}

By \reff{pg},  $(g(4)-g(2))/2$ is a Lipschitz constant for $g$ on $[0,2]$. However, by \reff{ub} and \reff{pg}, 
\eq{glip}
{
\frac{g(4)-g(2)}{2}\le \frac{g(4)-g(1)}2\le \frac{T^2-1}2. 
}
Thus 
\eq{el}
{
L= \frac{T^2-1}2\iq{is a Lipschitz constant for} g \i{on} [0,2].
}
Therefore
\eqns
{
&g\Lp\frac 1{1+\alpha}\Rp\le g(1)+L\frac{|\alpha|}{1+\alpha}\le 1+2 L|\alpha|,
\\& 
g\Lp\frac 1{1+2\alpha}\Rp\ge g(1)-L\frac{2|\alpha|}{1+2\alpha}\ge 1-4L|\alpha|.
}
Combining this information with \reff{bdel}, \reff{dtpb}, \reff{wa},  we find
\eq{punch}
{
\frac{h(x,x+2y)}{h(x,x+y)}\ge \Lambda\frac{1+2\alpha}{1+\alpha}\frac {g\Lp \frac 1{1+2\alpha}\Rp}{g\Lp \frac 1{1+\alpha}\Rp}\ge \Lambda\frac{1+2\alpha}{1+\alpha}\frac{1-4L|\alpha|}{ 1+2L|\alpha|}\ge \Lambda \frac{1-2\kappa}{1+\kappa}\frac{1-4L\kappa}{ 1+2L\kappa}.
}
 Choosing $\kappa$ sufficiently small, the right hand side can be made as close to $\Lambda$ as desired; hence it can be made larger than 2.   Let us review the restrictions used in this estimate.  They are  \reff{bdel}, \reff{pth}, \reff{ub}, and then a further restriction on $\kappa$ to make the quantity on the right of \reff{punch} as close to $\Lambda$ as we chose. 
 To have \reff{pth} hold in the presence of \reff{bdel}, it suffices to have 
  \eq{nep}
 {
 \varepsilon \le (1-2\kappa)r.
 }
In all, we require $\kappa\le 1/4$ to be small enough to guarantee that the right hand side of \reff{punch} is as close to $\Lambda$ as we specify (in particular, greater than 2), \reff{nep} and $|\alpha|\le \kappa.$  To finish, we need to  express these requirements  in terms of 
\eqn
{
y=\alpha x+\varepsilon x^\perp. 
}
We have
\eq{ed}
{
|\alpha|=|\ip{y,N(x)}|\le \no y
}
and 
\eq{angle}
{
\varepsilon=\no{y-\ip{y,N(x)}x}\le 2\no y. 
}
Thus 
\eq{ey}
{
\max(\varepsilon,|\alpha|)\le  2\no y;
}
therefore it suffices to have 
\eq{iny}
{
2\no y\le  \min(\kappa,(1-2\kappa)r). 
}
This condition is independent of the unit vector $x.$
 \hfill $\square$
 
 We turn to the demonstration that doubling in the tangent plane implies doubling, using the coordinates \reff{edc}.  First, by doubling in the tangent plane, by machinations as in \reff{dtpb},
 \eqns
 {
h(x,x+2y)&=(1+2\alpha)\Lp\no{x+\frac{2\varepsilon}{1+2\alpha}x^\perp}-\no x\Rp
\\&
= (1+2\alpha)h\Lp x,x+\frac{2\varepsilon}{1+2\alpha}x^\perp\Rp
\\& 
\le T(1+2\alpha)h\Lp x,x+\frac{\varepsilon}{1+2\alpha}x^\perp\Rp
\\&
= T(1+2\alpha)\Lp \no{x+\frac{\varepsilon}{1+2\alpha}x^\perp}-\no x\Rp
 }
 provided that \reff{bdel} holds. 
 
 On the other hand, 
 \eqns
 {
 h(x,x+y)&= \no{(1+\alpha)x+\varepsilon x^\perp}-(1+\alpha)\no x
 \\& =(1+\alpha)\Lp\no{x+\frac{\varepsilon}{1+\alpha} x^\perp}-\no x \Rp.
 }
 Now we may proceed as before to conclude that 
 \eq{douba}
 {
 h(x,x+2y)\le T\frac{1+2\alpha}{1+\alpha} \frac{1+4L|\alpha|}{1-2L|\alpha|}h(x,x+y),
  }
 where $L$ is given by \reff{el}, 
 so long as 
 \eq{mbd}
 {
 2L|\alpha|<1. 
 }
 We do not need to make $\kappa$ ``sufficiently small" here; \reff{bdel}, \reff{mbd} were the only restrictions employed.   We atomize this a bit more.  Let \reff{gdub} hold and $L$ be given by \reff{el}. Then \reff{douba} holds provided that
 \eq{doubb}
 {
 |\alpha|=|\ip{y,N(x)}|\le \min(1/4,1/(2L)). 
 }
 
The proof that balanced in the tangent plane implies balanced is a simple variation of arguments already given.  The result is   
\eq{fbal}
{
h(x,x+y)\le K\frac{1+\alpha}{1-\alpha}\frac {1+2L|\alpha|}{1-2L|\alpha|}h(x,x-y),
}
provided that \reff{doubb} holds. The estimates \reff{douba}, \reff{fbal} remain valid if $x$ is not a unit vector and \reff{doubb} is replaced by
\eq{doubg}
{
|\alpha|=|\ip{y,N(x)}|\le \no x\min(1/4,1/(2L)). 
}
\hfill$\square$ 
 
  \subsection{\label{lpc} The Norm $\nop{\cdot}$ in the Tangent Plane}
  We turn to the proofs that $\nop{\cdot}$ has all the properties above.
  
  \begin{theorem}\label{lptp} Let $1<p<\infty.$ Then $\nop{\cdot}$ is geometrically convex and doubling and balanced in the tangent plane.
  \end{theorem}
  
  The heart of the proof of the theorem is the following lemma about a function of two real variables. 
\begin{lemma}\label{onev} Let $1<p<\infty$ and $f:\R^2\ra\R$ be given by
\eq{rv}
{
f(x,y):=|x+y|^p-|x|^p-yp|x|^{p-1}\sign(x).
}
Then the following assertions hold:
\begin{itemize}
\item[(a)]
  There is a constant  $\Lambda>2$ such that
\eqn
{
f(x,2y)\ge \Lambda f(x,y) \iq{for}x,y\in\R. 
}
\item[(b)] There is a constant $T>0$ such that 
\eqn
{
f(x,2y)\le T f(x,y) \iq{for}x,y\in\R. 
}
\item[(c)] There is a constant $K>0$ such that 
\eqn
{
f(x,y)\le K f(x,-y) \iq{for}x,y\in\R. 
}
\end{itemize} 
\end{lemma}
\begin{proof}
We begin with the proof of (a), which is elementary, but perhaps not obvious. The claims (b) and (c) follow from  arguments used to establish (a), but  are much less subtle, and their proofs reduce to remarks. 

First note that since 
\eqn
{
f(0,2y)=2^p|y|^p=2^pf(0,y),
}
we may assume that $x\not=0,$ while $f(x,y)=f(-x,-y)$ shows that we may assume $x>0$ without loss of generality.  Dividing the inequality claimed in (a) by $x^p$ and putting $z=y/x,$  what we have to show is that there exists $\Lambda>2$ such that
\eq{claima}
{
 |1+2z|^p-1-p2z\ge\Lambda (|1+z|^p-1-pz)
}
for $z\in\R.$ 

Let 
\eq{dh}
{
g(z):=|1+z|^p-1-pz.
}
Note that $g(0)=0, g'(0)=0, g''(0)=p(p-1)\not=0.$ Thus 
\eq{zero}
{
\lim_{z\rightarrow 0}\frac{g(2z)}{g(z)}=4.
}
Clearly 
\eq{infty}
{
\lim_{z\rightarrow \pm\infty}\frac{g(2z)}{g(z)}=2^p.
}
Thus we have the desired inequality \reff{claima} for $z$ near 0 and near $\pm\infty,$ with any constant $\Lambda$ slightly less than $\min(4,2^p).$  Moreover,  it follows that if
\eq{claimc}
{
\inf_{z\not=0}\frac{g(2z)}{g(z)} \le 2,
}
then the inf is attained and there must be a point $z\in\R,$ $z\not=0,$ such that 
\eqns
{
g(2z)-2g(z)&=|1+2z|^p-1-2pz-2(|1+z|^p-1-pz)
\\&
= |1+2z|^p-2|1+z|^p+1=0. 
}
However, this function is positive for small $z\not=0$  and near $\pm\infty$ by the above, and its derivative is
\eqn
{
d(z):=2p(|1+2z|^{p-1}\sign(1+2z)-|1+z|^{p-1}\sign(1+z)),
}
which is continuous.  Clearly this $d(z)$ does not vanish unless 
\eqn
{
|1+2z|=|1+z|.
}
This last equation has only the solutions $z=0$ and $z=-2/3.$ Note that $d(-2/3)<0.$ Since $d$ can only change sign at a zero, and $d(z)$ is positive near $+\infty,$ it follows that $d(z)<0$ for $z<0$ and $d(z)>0$ for $z>0.$ Therefore the only zero of $g(2z)-2g(z)$ is $z=0,$ and (a) is proved. 

The assertions (b) and (c) yield to the first part of the  arguments above, as  one only needs to check them for $z$ near 0 and near infinity as 0 is the only zero of $g$. That is, the assertions amount to the statements that $g(2z)/g(z)$ and $g(z)/g(-z)$ are bounded. 
\end{proof}
\removelastskip
\proofof{Theorem \ref{lptp}.} We begin with the proof that $\nop{\cdot}$ is geometrically convex in the tangent plane.  To start, note that in this case that the $j$th component of $D\nop{x}^p$ is
\eqn
{
p|\xs j|^{p-1}\sign(\xs j),
}
while 
\eqn
{
D\nop x^p=p\nop{x}^{p-1}D\nop x. 
}
It follows that
\eq{nlp}
{
D\nop{x}=\frac1{\nop x^{p-1}}\Lp|\xs 1|^{p-1}\sign(\xs 1),\ldots,|\xs n|^{p-1}\sign(\xs n) \Rp.
}
Let $x,y\in\Rn,$ $x=(\xs1,\ldots,\xs n),$  etc. By Lemma \ref{onev} (a),  we have a $\Lambda>2$ such that 
\eqn
{
\sum_{i=1}^n (|\xs i+2\ys i|^p - |\xs i|^p-2\ys ip|\xs i|^{p-1}\sign(\xs i))\ge \Lambda \sum_{i=1}^n (|\xs i+\ys i|^p - |\xs i|^p-\ys ip|\xs i|^{p-1}\sign(\xs i)).
}
If we assume that 
\eq{tanp}
{
\ip{y,N(x)}=0;\i{equivalently,} \sum_{i=1}^n\ys i|\xs i|^{p-1}\sign(\xs i)=0,
}
i.e., $y$ is in the tangent plane, then we have
\eq{lpp}{
\nop{x+2y}^p-\nop x^p\ge \Lambda (\nop {x+y}^p-\nop x^p). 
}
On the other hand,  for each $\varepsilon>0$ there is an $r_\varepsilon>0$ such that
\eq{sest}{
p(1-\varepsilon)\le \frac {s^p-1}{s-1}\iq{and}\frac1p(1-\varepsilon)\le \frac {s-1}{s^p-1}
}
for
\eqn
{
|s-1|\le r_\varepsilon.
}
Suppose $\nop x=1.$ Then we use \reff{lpp} and \reff{sest} to find
\eqns
{
\nop{x+2y}-1&=(\nop{x+2y}^p-1)\frac{\nop{x+2y}-1}{\nop{x+2y}^p-1}
\\& \ge\Lambda(\nop{x+y}^p-1)\frac{\nop{x+2y}-1}{\nop{x+2y}^p-1}
\\&
=\Lambda (\nop{x+y}-1)\frac{\nop{x+y}^p-1 }{\nop{x+y}-1}\frac{\nop{x+2y}-1}{\nop{x+2y}^p-1}
\\&\ge\Lambda (\nop{x+y}-1)(1-\varepsilon)^2
}
for 
\eq{nxy}
{
1-r_\varepsilon\le\nop{x+y}, \nop{x+2y}\le 1+ r_\varepsilon.
}
It suffices that $\nop y\le r_\varepsilon/2;$ we may also choose $\varepsilon$ so that $\Lambda(1-\varepsilon)^2>2.$

The proofs that $\nop{\cdot}$ is balanced and doubling in the tangent plane run similarly.  For example, to show that it is doubling in the tangent plane, Lemma \ref{onev} (b) provides a constant $T$ such that 
\eqn
{
\sum_{i=1}^n (|\xs i+2\ys i|^p - |\xs i|^p-2\ys ip|\xs i|^{p-1}\sign(\xs i))\le T \sum_{i=1}^n (|\xs i+\ys i|^p - |\xs i|^p-\ys ip|\xs i|^{p-1}\sign(\xs i));
}
so if \reff{tanp} holds, we conclude that 
\eqn
{
\nop{x+2y}^p-\nop x^p\le T(\nop{x+y}^p-\nop x^p).
}
We may continue, as in the proof of geometric convexity in the tangent plane.  Checking the assertion that $\nop{\cdot}$ is balanced in the tangent plane is entirely similar, using Lemma \ref{onev} (c). 
\hfill $\square$

\section{Examples in the Case of the $p$-norm} \label{cexlp}
In this section we show that the H\"older exponent of Proposition \ref{holderp} is sharp for the $p$-norms.  We do this by producing $e, \be, m$ for which \reff{ren} holds and \reff{refa} is basically an equality, up to constants.

Let $n=3$ and $p>2.$ Put 
\eqs{php}
{
e:=(\delta, x,-x),\  \be:=(-\delta, x,-x),\ 
m:=(0,y,y)
}
where $\delta, x, y$ are positive numbers.  The conditions of \reff{ren} (with $L=1$) amount to
\eq{cref}{
1=\delta^p+2x^p,\ 1\le (x-y)^p+(x+y)^p.
}
where we presciently assume in the writing that $y<x,$ which is justified below. We are interested in small $\delta.$

We require information about the solutions of
\eq{solve}
{
1= (x-y)^p+(x+y)^p.
}
This may be rewritten as
\eqn
{
\frac 1{x^p}\le (1-r)^p+(1+r)^p, r=y/x. 
}
which we consider in the sharpened form
\eq{esolve}
{
2+\varepsilon=f(r)=(1-r)^p+(1+r)^p.
}
To match up with \reff{solve}, we would take, via \reff{cref},  
\eqn
{
\varepsilon =\frac 1{x^p}-2=\frac{2\delta^p}{1-\delta^p}.
}
Note that $\varepsilon>0$ is small if $\delta$ is small.  
Since we use it again later, we record some elementary facts about \reff{esolve} in a lemma.

\begin{lemma}\label{cexl} Let $1<p<\infty.$ If $0\le \varepsilon\le 2^p-2,$  then \reff{esolve} has a unique solution $r=g(\varepsilon)$ satisfying $0\le r\le 1.$ Moreover, $g$ is continuous, strictly increasing, differentiable on $(0,2^p-1),$ and satisfies
$g(0)=0, g(2^p-2)=1.$
\end{lemma}
\begin{proof}
The existence of a solution is guaranteed by the intermediate value theorem and $f(0)=0,$ $f(1)=2^p.$   Next, $f'(r)=p(-(1-r)^{p-1}+(1+r)^{p-1})>0$ for $0<r\le 1,$ so $f$  is strictly increasing. Solutions are therefore unique, and $g$ is well defined, strictly increasing and differentiable on the open interval by the implicit function theorem. 
\end{proof}

It follows that \reff{solve} has a unique solution $y\ge 0$ if $\delta$ is small, and $y/x$ is then small. 
So long as $y$ is small compared to $x,$ as it will be by the preceding remarks, we then have, by Taylor approximation, 
\eq{large}
{
(x-y)^p+(x+y)^p\approx 2x^p+p(p-1)x^{p-2}y^2. 
}
Thus the second condition of \reff{cref} is satisfied with  $y$ such that
\eq{y}
{
y^2\approx \frac{1-2x^p}{p(p-1)x^{p-2}}=\frac{\delta^p}{p(p-1)x^{p-2}}.
}
Since
\eq{cex}
{
\nop{e-\be}= 2\delta\iq{while}\nop{m}=2y\approx C \delta^{p/2},
}
Thus $\no{e-\be}$ and $\no m^{2/p}$ are comparable, verifying the sharpness of the exponent in \reff{holder} in this case.  
 
 To continue, we treat $1<p<2.$ In this case, we set 
 \eqn
 {
 e:=(x-\delta, x+\delta, 0), \ \be:=(x+\delta, x-\delta, 0), \ m:=(0,0,y). 
 }
 The conditions \reff{ren} (with $L=1$) become 
 \eq{conds}
 {
 1=(x-\delta)^p+(x+\delta)^p, \ 1\le 2x^p+y^p.
  }
Rewriting the first relation as 
\eqn
{
\frac 1{x^p}=(1-r)^p+(1+r)^p,\ r=\delta/x,
}
Lemma \ref{cexl} yields $\delta=xg(1/x^p-2)$ so long as
\eqn
{
0\le \frac 1{x^p}-2\le 2^p-2,
}
or 
\eqn
{
\frac1{2^p}\le x^p\le \frac 1{2}.
}
Moreover, $\delta\ra 0$ as $x^p$ increases to $1/2;$ in particular, $\delta$ is small compared to $x.$ Then the first relation of \reff{conds} tells us that
\eqn
{
1\approx 2x^p+p(p-1)x^{p-2}\delta^2
}
or
\eqn
{
\delta^2\approx \frac{1-2x^p}{p(p-1)x^{p-2}}.
}
Solving the second relation of \reff{conds} as an equality, we have
\eqn
{
y^p=1-2x^p.
}
Thus, as $x^p$ increases to $1/2$, we have that $y^p$ and $\delta^2$ are comparable, so 
\eqn
{
\nop{e-\be}=22^{1/p}\delta, \nop{m}=y
}
implies that $\nop{e-\be}$ is comparable to $\nop m^{p/2},$ verifying the sharpness of the exponent in this case. 

\renewcommand{\sc}{}

\small
\noindent Acknowledgements: The authors were partially supported by  NSF Grants DMS-0654267 and DMS-0400674, respectively.\\ The authors thank Carl de Boor for numerous helpful comments on this manuscript. \\
The second author thanks the Department of Mathematics, University of Texas, for its hospitality on several visits during this work. 
 
\end{document}